\documentclass[12pt]{article}
\usepackage{amsmath,amssymb,amsthm}
\usepackage{graphicx}
\usepackage[utf8x]{inputenc}
  \newtheorem{theorem}{Theorem}
 
 \newtheorem{claim}[theorem]{Claim}
 \newtheorem{lemma}[theorem]{Lemma}
 \newtheorem{proposition}[theorem]{Proposition}
 
 {\theoremstyle{definition}

 }
\newcommand{\N}{\ensuremath{\mathbb N}} 
\newcommand{\Z}{\ensuremath{\mathbb Z}} 
\usepackage{buczo}
\usepackage{tempdwe}

\renewcommand{\lll}{\lambda}
\title{Ergodic averages with prime divisor weights in $L^{1}$}
\author{Zolt\'an Buczolich\thanks{
Research supported by National Research, Development and Innovation Office--NKFIH, Grant 104178.
\newline\indent {\it Mathematics Subject
Classification:} Primary : 37A30; Secondary : 11A25, 28D05, 37A05.
\newline\indent {\it Keywords:} Ergodic averages, pointwise convergence,
maximal inequality, number of distinct prime factors.},
Department of Analysis, E\"otv\"os Lor\'and\\
University, P\'azm\'any P\'eter S\'et\'any 1/c, 1117 Budapest, Hungary\\
email: buczo@cs.elte.hu\\
{\tt www.cs.elte.hu/\hbox{$\sim$}buczo}
}
\date{\today}
\begin{document}
\maketitle

\medskip


\begin{abstract}
We show that $\ooo(n)$ and $\OOO(n)$, 
the number of distinct prime factors of $n$ and
the number of distinct prime factors of $n$ counted according to multiplicity are good weighting functions for the pointwise ergodic theorem in $L^{1}$. That is, if $g$ denotes one of these functions
and $S_{g,K}=\sum_{n\leq K}g(n)$
then
for every ergodic dynamical system $(X,\caa,\mu,\tttt)$ and every $f\in L^{1}(X)$
$$\lim_{K\to \oo} \frac{1}{S_{g,K}}\sum_{n=1}^{K} g(n)f(\tttt^{n}x)=\int_{X}fd\mu \text{ for $\mu$ a.e. }x\in X.
$$ 

This answers a question raised by C. Cuny and M. Weber
who showed this result for $L^{p}$, $p>1$.
  \end{abstract}


\section{Introduction}
In \cite{[CW]} C. Cuny and M. Weber investigated whether some arthimetic weights
are good weights for the pointwise ergodic theorem in $L^{p}$.  
In this paper we show that the prime divisor functions $\ooo$
and $\OOO$  are both good weights for the $L^{1}$ pointwise ergodic
theorem. The same fact for the spaces $L^{p}$, $p>1$ was proved  in \cite{[CW]} and our paper answers a question raised in that paper.
Recall that if $n=p_{1}^{\aaa_{1}}\cdots p_{k}^{\aaa_{k}}$
then $\ooo(n)=k$ and $\OOO(n)=\aaa_{1}+...+\aaa_{k}$.
We  denote by $g$ one of these functions.
Given $K$ we put $$S_{g,K}=\sum_{n\leq K}g(n).$$
We suppose that $(X,\caa,\mu)$ is a measure space and
$\tttt:X\to X$ is a measure preserving ergodic transformation.
Given $f\in L^{1}(X)$ we consider the $g$-weighted ergodic averages
\begin{equation}\label{*2*a}
\cam_{g,K}  f(x)=\frac{1}{S_{g,K}}\sum_{n=1}^{K} g(n)f(\tttt^{n}x).
\end{equation}
We show that for $g=\ooo$, or $\OOO$ these averages $\mu$ a.e. converge
to $\int_{X}fd\mu$, that is $g$ is a good universal weight for the pointwise
ergodic theorem in $L^{1}$. See Theorem \ref{*thdwe*}.

For some similar ergodic theorems with other weights like the M\"obius function,
 or its absolute value, or the Liouville function we refer to the papers of
 El Abdalaoui, Kułaga-Przymus,  Lemańczyk and de la Rue, \cite{[EAKLR]}, and of Rosenblatt and Wierdl \cite{[RW]}.

\section{Preliminary results}
We recall Theorem 430 from p. 72 of \cite{[HW]}
\begin{equation}\label{HW*1}
\sum_{n\leq K}\ooo(n)=K\log\log K+B_{1}K+o(K)\text{ and }
\end{equation}
\begin{equation}\label{HW*2}
\sum_{n\leq K}\OOO(n)=K\log\log K+B_{2}K+o(K).
\end{equation}
Hence, for both cases we can assume that there exists a constant
$B$ (which depends on whether $g=\ooo$, or $g=\OOO$) such that
\begin{equation}\label{*sumg*}
\sum_{n\leq K}g(n)=K\log\log K\Big (1+\frac{B}{\log\log K}+\frac{o(K)}{K\log\log K}\Big ).
\end{equation}
From this it follows that  there exists $C_{g}>0$  such that for all $K\in\N$
\begin{equation}\label{*3*a}
\Big ( \sum_{n\leq K} g(n) \Big )^{\lf \log\log K \rf}=(S_{g,K})^{\lf \log\log K \rf}>
C_{g}(K\lf \log\log K \rf)^{\lf \log\log K \rf}.
\end{equation}

We need some information about the distribution of the functions $\ooo$ and $\OOO$. We use (3.9) from p. 689 of  \cite{[Nor]} by K. K. Norton which is based on a result of Halász \cite{[Ha]}
which is cited as (3.8) Lemma in \cite{[Nor]}.
Next we  state (3.9) from \cite{[Nor]} with $\ddd=0.1$ and $z=2-\ddd=1.9.$
\begin{proposition}\label{*lemnor*}
There exists a constant $\widetilde{C}_H$  such that  for every $K\geq 1$
\begin{equation}\label{*nor}\sum_{n\leq K} 1.9^{\ooo(n)}\leq
\sum_{n\leq K} 1.9^{\OOO(n)}\leq
\widetilde{C}_H K\exp(0.9\cdot E(K)), \text{ where }
\end{equation}
$E(K)=\sum_{p\leq K}\frac{1}{p}.$
\end{proposition} 

Recall that by Theorem 427 in \cite{[HW]}
\begin{equation}\label{*prth*}
E(K)=\sum_{p\leq K}\frac{1}{p}=\log\log K+B_{1}+o(1).
\end{equation}
The constant $B_{1}$ is the same which appears in \eqref{HW*1}.
The way we will use this is the following: there exists a constant
$C_{P}$  such that  for $K>3$
\begin{equation}\label{*prth2*}
E(K)=\sum_{p\leq K}\frac{1}{p}<C_{P}\log\log K.
\end{equation}
Combining this with \eqref{*nor} we obtain that for $g=\ooo$,
or $\OOO$ we have for $K>3$
\begin{equation}\label{*nor2*}
\sum_{n\leq K} 1.9^{g(n)}<\widetilde{C}_H \cdot K\cdot \exp(0.9 \cdot C_{P}\log\log K)
\leq C_{H}\cdot K\exp (0.9\cdot C_{P}\lf \log\log K \rf), 
\end{equation}
with a suitable constant $C_{H}$ not depending on $K$.

In \cite{[CW]} a result of Delange \cite{[Del]} was used to deduce Theorem 2.7 in \cite{[CW]}. The result of Delange is the following
\begin{theorem}\label{*th27CW*}
For every $m\geq 1$ we have 
$$\sum_{n\leq K} g(n)^{m}=K(\log\log K)^{m}+O(K(\log\log K)^{m-1}).$$
\end{theorem}

We were unable to use this result since  the constant in
$O(K(\log\log K)^{m-1})$  cannot be chosen not depending on  $m\geq 1$.

Hence we use \eqref{*nor2*} in the proof of the following lemma.
\begin{lemma}\label{*lemomax*}
There exists a constant $C_{\OOO,max} $  such that for all $K\geq 16$
\begin{equation}\label{*dwe7*a}
\sum_{n\leq K}\ooo(n)^{\lf \log\log K \rf}\leq
\sum_{n\leq K}\OOO(n)^{\lf \log\log K \rf}<  K(C_{\OOO,max} \lf \log\log K \rf)^{\lf \log\log K \rf}.
\end{equation}
\end{lemma}

We remark that the assumption $K\geq 16$ implies that
$\log\log K>1.01>1.$

\begin{proof}
Since $\ooo(n)\leq \OOO(n) $ the first  inequality is obvious in \eqref{*dwe7*a},

We assume that $K\geq 16$ is fixed and for ease of notation we put 
$\nnn=\lf \log\log K \rf.$
Set
\begin{equation}\label{*dwe7*b}
N_{l,K}=\{ n\leq K:2^{l}\nu \leq \OOO(n)<2^{l+1}\nnn \}.
\end{equation}
By \eqref{*nor2*}
$N_{l,K}\cdot 1.9^{2^{l}\nnn}<C_{H}K\exp(0.9\cdot C_{P}\nnn).$
This implies that
\begin{equation}\label{*dwe8*b}
N_{l,K}<C_{H}K\cdot \exp((0.9 \cdot C_{P}-2^{l}\log 1.9)\nnn).
\end{equation}
Since $\log 1.9>0.6$ we can choose $l_{0}$  such that for
$l\geq l_{0}$
\begin{equation}\label{*dwe9*a}
0.9\cdot C_{P}-2^{l}\log 1.9+(l+1)\log 2
<-0.5\cdot 2^{l}=-2^{l-1}.
\end{equation}
From \eqref{*dwe8*b} and \eqref{*dwe9*a} we infer
\begin{equation}\label{*dwe8*a}
\sum_{n\leq K} \OOO(n)^{\nnn}<
K\cdot (2\nnn)^{\nnn}+\sum_{l=1}^{\oo}
N_{l,K}(2^{l+1}\nnn)^{\nnn}\leq
\end{equation}
$$K\cdot (2\nnn)^{\nnn}+\sum_{l=1}^{l_{0}-1}K(2^{l+1}\nnn)^{\nnn}+\sum_{l=l_{0}}^{\oo}C_{H}K \nnn^{\nnn}\exp(((\log 2^{l+1})+0.9 C_{P}-2^{l}\log 1.9)\nnn)<$$
(using \eqref{*dwe9*a} with a suitable constant $C_{\OOO,1}>2$ we obtain)
$$C_{\OOO,1}^{\nnn}K\nnn^{\nnn}+ \sum_{l=l_{0}}^{\oo}C_{H}K\nnn^{\nnn}\exp(-2^{l-1}\nnn)<$$
(recalling that $\nnn=\lf \log\log K \rf\geq \lf \log\log 16 \rf=1$, with a suitable constant $C_{\OOO,max}$ we have)
$$K\nnn^{\nnn}\Big (C_{\OOO,1}^{\nnn}+C_{H}\sum_{l=l_{0}}^{\oo}\exp({-2^{l-1}})\Big )< C_{\OOO,max}^{\nnn}  K\nnn^{\nnn}=$$
$$ K(C_{\OOO,max}  \lf \log\log K \rf)^{\lf \log\log K \rf}.$$
\end{proof}

We need the following (probably well-known) elementary inequality
to which we could not find a reference and hence provided the short proof.
\begin{lemma}\label{*lemperm*}
Suppose $K,\nnn\in\N$, $b_{1},...,b_{K}$ are nonnegative numbers
and we have permutations
$\pi_{j}:\{ 1,...,K \}\to \{ 1,...,K \},$ $j=1,...,\nnn$. Then
\begin{equation}\label{*dwe11b*a}
b_{\pi_{1}(1)}\cdots b_{\pi_{\nnn}(1)}+
...+
b_{\pi_{1}(K)}\cdots b_{\pi_{\nnn}(K)}\leq
b_{1}^{\nnn}+...+b_{K}^{\nnn}.
\end{equation}
\end{lemma}

\begin{proof}
Without limiting generality we can suppose that $0\leq b_{1}\leq ... \leq b_{K}$.
First observe that if $A> B\geq 0$ and $C> D\geq 0$ then
\begin{equation}\label{*dwe11b*b}
\text{
 from
$(A-B)(C-D)\geq 0$ it follows that $AC+BD\geq AD+BC.$
}
\end{equation}
Set $\pi_{j,1}(k)=\pi_{j}(k)$ for $j=1,...,\nnn$ and $k=1,...,K$.
If $\pi_{j,l}$ is defined for an $l\in\N$ then set
$${\mathfrak M}_{l}^{*}=\max_{k}b_{\pi_{1,l}(k)}\cdots b_{\pi_{\nnn,l}(k)}.$$
We want to define a sequence of permutations such that for every $l$
\begin{equation}\label{*dwe11c*a}
b_{\pi_{1,l-1}(1)}\cdots b_{\pi_{\nnn,l-1}(1)}
+...+
b_{\pi_{1,l-1}(K)}\cdots b_{\pi_{\nnn,l-1}(K)}
\leq
\end{equation}
$$b_{\pi_{1,l}(1)}\cdots b_{\pi_{\nnn,l}(1)}
+...+
b_{\pi_{1,l}(K)}\cdots b_{\pi_{\nnn,l}(K)}.$$
Suppose that ${\mathfrak M}_{l}^{*}<b_{K}^{\nnn}.$
Select $k^{*}$
 such that  ${\mathfrak M}_{l}^{*}=b_{\pi_{1,l}(k^{*})}\cdots b_{\pi_{\nnn,l}(k^{*})}$.
 Then we can select $j^{*}$  such that $b_{\pi_{j^*,l}(k^{*})}<b_{K}$
 and $k^{**}$  such that $b_{\pi_{j^*,l}(k^{**})}=b_{K}.$
 Set $A=b_{K}=b_{\pi_{j^*,l}(k^{**})}$, $B=b_{\pi_{j^*,l}(k^{*})}$,
 $C=b_{\pi_{1,l}(k^{*})}\cdots b_{\pi_{\nnn,l}(k^{*})}/B=
 {\mathfrak M}_{l}^{*}/B$
 and
 $D=b_{\pi_{1,l}(k^{**})}\cdots b_{\pi_{\nnn,l}(k^{**})}/A$.
 Then $A>B\geq 0$ and $C> D\geq 0.$
 Set $\pi_{j^{*},l+1}(k^{**})=\pi_{j^{*},l}(k^{*})$,
 $\pi_{j^{*},l+1}(k^{*})=\pi_{j^{*},l}(k^{**})$, and for any other
 $j$ and $k$
set $\pi_{j,l+1}(k)=\pi_{j,l}(k)$. From \eqref{*dwe11b*b} it follows that
\eqref{*dwe11c*a} holds with $l$ replaced by $l+1$ and
${\mathfrak M}^{*}_{l+1}>{\mathfrak M}^{*}_{l}$. Hence in finitely many steps there is $l_{1}$
 such that ${\mathfrak M}_{l_{1}}^{*}=b_{K}^{\nnn}$.

After step $l_{1}$ arguing as above we can still define the permutations $\pi_{j,l}$
so that \eqref{*dwe11c*a} holds at each step and can reach a step $l_{2}$
 such that ${\mathfrak M}_{l_{2}}^{*}=b_{K}^{\nnn}$ and the second largest
 term among $b_{\pi_{1,l_{2}}(k)}\cdots b_{\pi_{\nnn,l_{2}}(k)}$,
 $k=1,...,K$ equals $b_{K-1}^{\nnn}.$
Repeating this procedure one can obtain \eqref{*dwe11b*a}. 
\end{proof}

We will use the transference principle and hence we need to consider
functions on the integers. Suppose $\fff:\Z\to [0,+\oo)$
is a function on the integers with compact/bounded support.
Again $g$ will denote $\ooo$, or $\OOO$.
Put
$$M_{g,K} \fff(j)=\frac{1}{S_{g,K}}\sum_{n=1}^{K} g(n)\fff(j+n)\text{ for }j\in\Z.$$

First we prove a ``localized" maximal inequality.

\begin{lemma}\label{*lemlocmax*}
There exists a constant $C_{g,max}>0$  such that for any $\fff:\Z\to [0,+\oo)$,
$K\geq 16$ and $k\in\Z$
\begin{equation}\label{*dwe12*a}
\sum_{j=1}^{K}(M_{g,K} \fff(k+j))^{\lf \log\log K \rf}\leq
\Big (\sum_{j=2}^{2K}\fff(k+j)\Big )
\Big ( \frac{C_{g,max}}{K}\sum_{j=2}^{2K} \fff(k+j) \Big )^{\lf \log\log K \rf -1}.
\end{equation}
\end{lemma}

\begin{proof}
Without limiting generality we can suppose that $k=0$ and $K\geq 16$
is fixed.
We use again the notation $\nnn=\nnn_{K}=\lf \log\log K \rf$.
We put
\begin{equation}\label{*dwe13*a}
\tg(n)=\tg_{K}(n)=\twocase {g(n)}{\text{if $1\leq n\leq K$}}
{0}{\text{otherwise.}}
\end{equation}
We need to estimate
$$\sum_{j=1}^{K}\Big (\frac{1}{S_{g,K}} \sum_{n=1}^{K} g(n)\fff(j+n)\Big )^{\nnn}=$$
$$\frac{1}{S_{g,K}^{\nnn}}\sum_{j=1}^{K}\sum_{n_{1}=1}^{K}...
\sum_{n_{\nnn}=1}^{K}g(n_{1})\cdots g(n_{\nnn})\cdot 
\fff(j+n_{1})\cdots \fff(j+n_{\nnn})=$$
$$
\frac{1}{S_{g,K}^{\nnn}}\sum_{n'=1}^{K}\sum_{j_{1}=2}^{2K}...
\sum_{j_{\nnn}=2}^{2K}\fff(j_{1})\cdots \fff(j_{\nnn})\cdot 
\tg(n')\tg(n'+j_{2}-j_{1})\cdots \tg(n'+j_{\nnn}-j_{1})=
$$
$$
\frac{1}{S_{g,K}^{\nnn}}\sum_{j_{1}=2}^{2K}...
\sum_{j_{\nnn}=2}^{2K}\fff(j_{1})\cdots \fff(j_{\nnn})
\cdot \sum_{n'=1}^{K} 
\tg(n')\tg(n'+j_{2}-j_{1})\cdots \tg(n'+j_{\nnn}-j_{1})\leq
$$
(using Lemma \ref{*lemperm*} and \eqref{*dwe13*a})
$$
\frac{1}{S_{g,K}^{\nnn}}\sum_{j_{1}=2}^{2K}...
\sum_{j_{\nnn}=2}^{2K}\fff(j_{1})\cdots \fff(j_{\nnn})
\cdot \sum_{n'=-K+2}^{2K-1} 
(\tg(n'))^{\nnn}=
$$
$$
\frac{1}{S_{g,K}^{\nnn}}\sum_{j_{1}=2}^{2K}...
\sum_{j_{\nnn}=2}^{2K}\fff(j_{1})\cdots \fff(j_{\nnn})
\cdot \sum_{n'=1}^{K} 
(g(n'))^{\nnn}\leq
$$
(by using Lemma \ref{*lemomax*})
$$
K\cdot C_{\OOO,max}^{\nnn}\nnn^{\nnn} \frac{1}{S_{g,K}^{\nnn}}
\Big (\sum_{j=2}^{2K} \fff(j)\Big )^{\nnn}<
$$
(by \eqref{*3*a})
$$
K\cdot C_{\OOO,max}^{\nnn}  \nnn^{\nnn} \frac{1}{C_{g}(K\nnn)^{\nnn}}
\Big (\sum_{j=2}^{2K} \fff(j)\Big )^{\nnn}<
$$
(with a suitable constant $C_{g,max}>0$)
$$<
\Big (\sum_{j=2}^{2K} \fff(j)\Big )
\cdot 
\Big (C_{g,max}\frac{1}{K}\sum_{j=2}^{2K} \fff(j)\Big )^{\nnn-1}.
$$
\end{proof}


\section{Main result}

\begin{theorem}\label{*thdwe*}
For every ergodic dynamical system $(X,\caa,\mu,\tttt)$ and every $f\in L^{1}(X)$
\begin{equation}\label{*dwe2a*b}
\lim_{K\to \oo}\cam_{g,K} f(x)=\int_{X}fd\mu \text{ for $\mu$ a.e. }x\in X.
\end{equation}
\end{theorem}


\begin{proof}

By Theorem 2.5 and Remark 2.6 of \cite{[CW]} we know that 
$\ooo$ and $\OOO$ are good weights for the pointwise ergodic
theorem in $L^{p}$ for $p>1$. This means that we have a dense
set of functions in $L^{1}$ for which the pointwise
ergodic theorem holds. In Theorem 2.5  of \cite{[CW]} it is not stated 
explicitely that the limit function
of the averages $\cam_{g,K} f$ is $\int_{X}fd\mmm$, but from the proof
of this theorem it is clear that $\cam_{g,K} f$ not only converges a.e., but its
limit is indeed $\int_{X}fd\mmm$ (at least for $f\in L^{\oo}(\mmm)$).
Indeed, from (2.2) in  \cite{[CW]} it follows that $\cam_{g,K} f$ can be 
written as the sum of an ordinary Birkhoff-average of $f$ and an error term which
tends to zero as $K\to\oo$.

Hence 
by standard application of Banach's principle (see for example \cite{[P]} p. 91)
the following weak $L^1$-maximal inequality proves Theorem \ref{*thdwe*}.
\begin{proposition}\label{*propwmaxmu*}
There exists a constant $C_{max}$  such that for every ergodic dynamical
system $(X,\caa,\mmm,\tttt)$ for every $f\in L^{1}(\mmm)$ and $\lll\geq 0$
\begin{equation}\label{*dwe11*a}
\mmm\{ x:\sup_{K\geq 1}\cam_{g,K} f(x)>\lll \}\leq C_{max}\frac{||f||_{1}}{\lll}.
\end{equation}
\end{proposition}
\end{proof}


\begin{proof}[Proof of Proposition \ref{*propwmaxmu*}.]
By standard transference arguments, see for example \cite{[RW]} Chapter III,
it is  sufficient to establish a corresponding weak maximal inequality on the integers with $\lll=1$ for nonnegative functions with compact support.
Hence, this proof will be completed by Proposition \ref{*propmaxi*}
below.
\end{proof}

Thus we need to state and prove the following maximal inequality:
\begin{proposition}\label{*propmaxi*}
There exists a constant $C_{max}$
 such that for every $\fff:\Z\to [0,\oo)$ with compact support
 $$\# \{ j:\sup_{K\in \N} M_{g,K} \fff(j)>1 \}\leq C_{max}||\fff||_{\ell _{1}}.$$
\end{proposition}


Proposition \ref{*propmaxi*} can also be reduced further to the following Claim.
Set $M_{l}=M_{g, 2^{l}}$.

\begin{claim}\label{*twol*}
There exists a constant $C'_{max}$  such that for every $\fff:\Z\to [0,+\oo)$
with compact support
\begin{equation}\label{*dwe16*a}
\# \{ j:\sup_{l\in \N} M_{l} \fff(j)>1 \}\leq C_{max}'||\fff||_{\ell _{1}}.
\end{equation}
\end{claim}


\begin{proof}[Proof of Proposition \ref{*propmaxi*} based on Claim \ref{*twol*}.]
Given $K\in \N$ choose $l_{K}\in \N$ such that $2^{l_{K}-1}<K\leq 2^{l_{K}}$.
By \eqref{HW*1}, or \eqref{HW*2} 
there exists a constant $C_{R}>0$ not depending on $K$
 such that $S_{g,2^{l_{K}}} \leq  C_{R}S_{g,K}.$
 We have 
 $$1< M_{g,K}\fff(j)=\frac{1}{S_{g,K}}\sum_{j=1}^{K}g(n)\fff(j+n)\leq$$
 $$\frac{S_{g,2^{l_{K}}}}{S_{g,K}}\cdot \frac{1}{S_{g,2^{l_{K}}}}
 \sum_{n=1}^{2^{l_{K}}}g(n)\fff(j+n)\leq C_{R}
 M_{g,2^{l_{K}}}\fff(j).$$
 Hence, $1<M_{g,K}\fff(j)$ implies $\frac{1}{C_{R}}<M_{g,2^{l_{K}}}\fff(j)=M_{l_{K}}\fff(j).$
 
 For any $\tfff: \Z\to [0,+\oo)$ with compact support
 taking $\fff=C_{R}\tfff$ by Claim \ref{*twol*} we obtain
 $$\# \{ j:\sup_{K\in\N}M_{g,K}\tfff(j)>1 \}\leq
 \# \{ j:\sup_{l\in\N} M_{l}\fff(j)>1 \}\leq
 $$
 $$C_{max}'||\fff||_{\ell _{1}}=
 C_{max}'C_{R}||\tfff||_{\ell _{1}}.$$
\end{proof}

\begin{proof}[Proof of Claim \ref{*twol*}.] 
If $1\leq l\leq 4$ then consider the set
$E_{l}=\{ j: M_{l}\fff(j)>1 \}$
and the system of intervals
$\cai_{l}=\{ [j+1,j+2^{l}]\cap \Z: j\in E_{l} \}$.
Then $E_{l}+1\sse \cup_{I\in \cai_{l}}I$ and hence
$\# E_{l}\leq \# \cup_{I\in\cai_{l}}I.$
We can select a subsystem $\cai_{l}'\sse \cai_{l}$
 such that no point of $\Z$ is covered by more than two intervals
 belonging to $\cai_{l}'$ and $\cup_{I\in\cai_{l}'}I=\cup_{I\in\cai_{l}}I.$
 
Suppose $I=[j+1,j+2^{l}]\cap\Z \in \cai_{l}'\sse \cai_{l}$.
Then $M_{l} \fff(j)>1$ implies that 
$$1<\frac{1}{S_{g,2^l}}\sum_{n=1}^{2^{l}}g(n)\fff(j+n),$$
that is
$$S_{g,2^l}\leq \sum_{n=1}^{2^{l}}g(n)\fff(j+n)=\sum_{k\in I}g(k-j)\fff(k).$$
Thus
$$1\leq \frac{S_{g,2^l}}{\max_{k\leq 2^{l}}g(k)}\leq \sum_{k\in I}\fff(k).$$

If $l\leq 4$ then we have $\# I/16\leq 1 \leq \sum_{k\in I}\fff(k)$.
Since 
no point is covered by more than two intervals $I\in \cai_{l}'$, that is,
$\sum_{I\in\cai_{l}'}\chi_{I}(j)\leq 2$, ($j\in \Z$) we obtain that
for $l\leq 4$
$$\# E_{l}\leq \# \cup_{I\in \cai_{l}'}I\leq 32 ||\fff ||_{\ell _{1}}$$
and hence
\begin{equation}\label{*dwe20*a}
\# \{ j:\sup_{1\leq l\leq 4 }M_{l}\fff(j)>1 \}\leq 128 ||\fff ||_{\ell _{1}}.
\end{equation}

Next suppose that $l>4.$
We consider the dyadic intervals $(r2^{l},(r+1)2^{l}]\cap \Z$, $r\in \Z$.
We say that $r\in R_{l,+}$ if
\begin{equation}\label{*14*a}
\frac{1}{2^{l}}\sum_{j=r2^{l}+1}^{r2^{l}+2\cdot 2^{l}}\fff(j)>\frac{1}{100 \cdot C_{g,max}}.
\end{equation}
Otherwise, if $r\not \in R_{l,+}$ we say that $r\in R_{l,-}$.

For $r\in R_{l,-}$ we use Lemma \ref{*lemlocmax*} and the negation
of \eqref{*14*a}
to deduce that for $l>4$
\begin{equation}\label{*dwe21*a}
\sum_{j=1}^{2^{l}}(M_{l}\fff(r2^{l}+j))^{\lf \log\log 2^{l} \rf}
< \Big (\sum_{j=2}^{2\cdot 2^{l}}\fff(r2^{l}+j)\Big )
\cdot \Big (\frac{1}{100}\Big )^{\lf \log\log 2^{l} \rf-1}\leq
\end{equation}
$$
 100^{2}\Big (\sum_{j=2}^{2\cdot 2^{l}}\fff(r2^{l}+j)\Big )
\cdot \Big (\frac{1}{100}\Big )^{ \log\log 2^{l} }\leq
$$
$$
 100^{2}\Big (\sum_{j=2}^{2\cdot 2^{l}}\fff(r2^{l}+j)\Big )
\cdot \exp(-(\log 100)\cdot \log\log 2^{l})\leq $$ $$
 100^{2}\Big (\sum_{j=2}^{2\cdot 2^{l}}\fff(r2^{l}+j)\Big )
\cdot \frac{6}{l^{2}}, \text{ where we used that}
$$
 $4.61 \geq \log 100\geq 4.60517$ and $\log\log 2>-0.37$
implies that $$\exp(-(\log 100)\cdot \log\log 2^{l})=\exp(-(\log 100)((\log l) +\log\log 2))=$$ $$\exp(-(\log 100)\log\log 2)\cdot \exp(-(\log 100)\log l)<\frac{6}{l^{2}}.$$

Set $\cam^{*}_{l}=\{ j:M_{l}\fff(j)>1 \}$
and $\cam^{*}=\cup_{l}\cam_{l}^{*}.$

If $r\in R_{l,-}$ then by \eqref{*dwe21*a}
$$\# (\cam_{l}^{*}\cap (r2^{l},(r+1)2^{l}])\leq \sum_{j=1}^{2^{l}}
(M_{l}\fff(r2^{l}+j))^{\lf \log\log 2^{l} \rf}\leq$$
$$  6\cdot 100^{2}\cdot \frac{1}{l^{2}}
\Big (\sum_{j=2}^{2\cdot 2^{l}}\fff(r2^{l}+j)\Big ).$$
Hence
$$\# (\cam_{l}^{*}\cap \bigcup_{r\in R_{l,-}}(r2^{l},(r+1)2^{l}])\leq$$
$$12\cdot 100^{2}\frac{1}{l^{2}}||\fff_{\ell _{1}}||$$
and
\begin{equation}\label{*dwe23*a}
\#\Big ( \bigcup_{l}(\cam^{*}_{l}\cap \cup_{r\in R_{l,-}}(r2^{l},(r+1)2^{l}]) \Big )\leq 
12\cdot 100^{2} \frac{\pi^{2}}{6}||\fff_{\ell _{1}}||.
\end{equation}
On the other hand,
\begin{equation}\label{*26b***}
\cup_{l>4}\cup_{r\in R_{l,+}}(r2^{l},(r+1)2^{l}]\cap \Z\sse
\cup_{l>4}\cup_{r\in R_{l,+}}[r2^{l},(r+2)2^{l}]\cap \Z
.
\end{equation}
We can again select a subsystem $\cai_{+}^{*}$ of the intervals $\cai_{+}=
\{ [r2^{l},(r+2)2^{l}]:l>4,\  r\in R_{l,+} \}$
 such that 
 \begin{equation}\label{*11***a}
\sum_{I\in\cai_{+}^{*}}\chi_{I}(j)\leq 2\text{ for all }j\in \Z
 \text{ and }\cup_{I\in \cai_{+}}I=\cup_{I\in \cai_{+}^{*}}I.
  \end{equation}
From \eqref{*14*a} it follows that 
if
$[r2^{l},(r+2)2^{l}]=I\in \cai_{+}^{*}$ then
$$
C_{g,max}\cdot 400 \sum_{j\in I}\fff(j)> 4\cdot 2^{l}>\#(I\cap \Z).
$$
Thus, by \eqref{*11***a}
$$\# (\cup_{I\in \cai_{+}}I\cap \Z)=
\# (\cup_{I\in \cai_{+}^{*}}I\cap \Z)<C_{g,max}\cdot  800||\fff||_{\ell _{1}}.$$
Hence, by \eqref{*26b***}
$$\# \bigg (\cup_{l>4}\cup_{r\in R_{l,+}}(r2^{l},(r+1)2^{l}]\cap\Z\bigg ) \leq C_{g,max}\cdot 800||\fff||_{\ell _{1}}.$$
From this, \eqref{*dwe20*a} and \eqref{*dwe23*a} it follows that
$$\# \cam^{*}\leq (128+12\cdot 100^{2}\frac{\pi^{2}}{6}+800C_{g,max})
||\fff||_{\ell _{1}}=C_{max}'||\fff||_{\ell _{1}}.$$
\end{proof}



\begin{thebibliography}{99}


\bibitem{[CW]} C. Cuny and M. Weber, {\it Ergodic theorems with arithmetical weights}, Preprint:\verb!https://arxiv.org/abs/1412.7640!.

\bibitem{[Del]} H. Delange, {\it Sur des formules de Atle Selberg,} Acta Arith. {\bf 19}, 105-146, (1971).

\bibitem{[EAKLR]} E. H. El Abdalaoui, J. Kułaga-Przymus,   M. Lemańczyk and T. De La Rue, {\it The Chowla and the Sarnak
conjectures from ergodic theory point of view}, Preprint:\verb!https://arxiv.org/abs/1410.1673!. 
 





%

\bibitem{[Ha]} G. Hal\'asz, {\it On the distribution of additive and the mean values of multiplicative arithmetic
functions}, Studia Sci. Math. Hungar., {\bf 6}, 211-233,  (1971).

\bibitem{[HW]} G.H. Hardy and E. M. Wright,  {\it An introduction to the theory of numbers}. Sixth edition. Oxford University Press, Oxford, (2008). 


\bibitem{[Nor]} K. K. Norton,  {\it  On the number of restricted prime factors of an integer. I.} Illinois J. Math. {\bf 20}, no. 4, 681–705, (1976).

\bibitem{[P]} {K. Petersen,} {\it
Ergodic Theory},
Cambridge Studies in Advanced Mathematics 2,
Cambridge University Press, (1981).


\bibitem{[RW]} J. M. Rosenblatt and M. Wierdl, {\it Pointwise ergodic theorems via harmonic analysis.} Ergodic theory and its connections with harmonic analysis (Alexandria, 1993), 3–151, London Math. Soc. Lecture Note Ser., 205, Cambridge Univ. Press, Cambridge, (1995).





\end{thebibliography}
\end{document}